\documentclass[preprint]{elsarticle}

\usepackage{amsthm}
\usepackage{amsmath}
\usepackage{amssymb}
\usepackage{bm}
\usepackage{mathrsfs}

\usepackage{algorithmic}
\usepackage{algorithm}

\usepackage{color}

\usepackage{url}

\usepackage{supertabular}

\newtheorem{thm}{Theorem}

\newtheorem{lemma}[thm]{Lemma}

\theoremstyle{definition}

\theoremstyle{remark}

\numberwithin{thm}{section}

\DeclareMathAlphabet{\mathsfsl}{OT1}{cmss}{m}{sl}

\renewcommand{\phi}{\varphi}

\newcommand{\argmin}{\operatorname*{arg\; min}}

\newcommand{\rank}{\operatorname{rank}}

\newcommand{\diag}{\operatorname{diag}}

\newcommand{\psdge}{\succcurlyeq}

\newcommand{\dist}{\operatorname{dist}}

\newcommand{\bx}{\boldsymbol{x}}

\newcommand{\by}{\boldsymbol{y}}
\newcommand{\bZ}{\boldsymbol{Z}}

\newcommand{\bz}{\boldsymbol{z}}
\newcommand{\bX}{\boldsymbol{X}}
\newcommand{\bL}{\boldsymbol{L}}
\newcommand{\bK}{\boldsymbol{K}}

\newcommand{\bP}{\boldsymbol{P}}

\newcommand{\bSigma}{\boldsymbol\Sigma}

\newcommand{\bU}{\boldsymbol{U}}
\newcommand{\bV}{\boldsymbol{V}}

\def\reals{\mathbb{R}}
\def\bx{\boldsymbol{x}}

\def\bu{\boldsymbol{u}}

\def\b0{\mathbf{0}}
\def\bP{\boldsymbol{P}}

\def\bSigma{\boldsymbol\Sigma}

\def\bU{\boldsymbol{U}}

\def\bC{\boldsymbol{C}}
\def\bv{\boldsymbol{v}}
\def\bA{\boldsymbol{A}}
\def\bY{\boldsymbol{Y}}
\def\bB{\boldsymbol{B}}

\def\bH{\boldsymbol{H}}

\def\bI{\mathbf{I}}

\def\tr{\mathrm{tr}}
\def\rmL{\mathrm{L}}

\def\Sp{\mathrm{Sp}}
\def\Col{\mathrm{Col}}

\usepackage{algorithmic}
\usepackage{algorithm}

\usepackage{hyperref}

\journal{Linear Algebra and its Applications}

\begin{document}
\begin{frontmatter}

\title{Disentangling Orthogonal Matrices}

\author{Teng Zhang\fnref{myfootnote}}
\address{Department of Mathematics, University of Central Florida\\  Orlando, Florida 32816, USA}
\fntext[myfootnote]{teng.zhang@ucf.edu}

\author{Amit Singer\fnref{myfootnote1}}
\address{Department of Mathematics and The Program in Applied and Computational Mathematics (PACM), Princeton University\\ Princeton, New Jersey 08544, USA}
\fntext[myfootnote1]{amits@math.princeton.edu}





\begin{abstract}
Motivated by a certain molecular reconstruction methodology in cryo-electron microscopy, we consider the problem of solving a linear system with two unknown orthogonal matrices, which is a generalization of the well-known orthogonal Procrustes problem. We propose an algorithm based on a semi-definite programming (SDP) relaxation, and give a theoretical guarantee for its performance. Both theoretically and empirically, the proposed algorithm performs better than the na\"{i}ve approach of solving the linear system directly without the orthogonal constraints. We also consider the generalization to linear systems with more than two unknown orthogonal matrices.
\end{abstract}
\begin{keyword}
SDP relaxation \sep orthogonal Procrustes problem \sep cryo-EM
MSC[2010] 15A24\sep  90C22
\end{keyword}

\end{frontmatter}


\section{Introduction}
In this paper, we consider the following problem: given known matrices $\bX_1, \bX_2\in\reals^{N\times D}$ and unknown orthogonal matrices $\bV_1, \bV_2\in O(D)$, recover $\bV_1$ and $\bV_2$ from $\bX_3\in\reals^{N\times D}$ defined by  \begin{equation}\label{eq:original0}\bX_3=\bX_1\bV_1+\bX_2\bV_2.\end{equation}
A na\"{i}ve approach would be solving \eqref{eq:original0} while dropping the constraints of orthogonality on $\bV_1$ and $\bV_2$. This linear system has $N D$ linear constraints and $2D^2$ unknown variables, therefore, this approach can recover $\bV_1$ and $\bV_2$ when $N\geq 2D$. The question is, can we develop an algorithm that takes the constraints of orthogonality into consideration, so that it is able to recover $\bV_1$ and $\bV_2$ when $N<2D$, and more stably when the observation $\bX_3$ is contaminated by noise?

The associated least squares problem \begin{equation}\label{eq:original00}\min_{\bV_1, \bV_2\in O(D)} \|\bX_1\bV_1+\bX_2\bV_2-\bX_3\|_F^2\end{equation} can be considered as a generalization of the well-known orthogonal Procrustes problem~\cite{gower2004procrustes}:
\begin{equation}\label{eq:original01}
\min_{\bV\in O(D)} \|\bX_1\bV-\bX_2\|_F^2,
\end{equation}
with the main difference being that the minimization in \eqref{eq:original00} is over two orthogonal matrices instead of just one as in \eqref{eq:original01}. Although the orthogonal Procrustes problem has a closed form solution using the singular value decomposition,  problem \eqref{eq:original00} does not enjoy this property.


Still, \eqref{eq:original00} can be reformulated so that it belongs to a wider class of problems called the little Grothendieck problem~\cite{Afonso2013}, which again belongs to QO-OC (Quadratic Optimization under Orthogonality Constraints) considered by Nemirovski~\cite{Nemirovski2007}. QO-OCs have been well studied and include many important problems as special cases, such as Max-Cut~\cite{Goemans1995} and generalized orthogonal Procrustes \cite{Gower1975,Berge1977,Shapiro1988}  \[\min_{\bV_1, \ldots, \bV_n\in O(D)}\sum_{1\leq i,j\leq n}\|\bX_i\bV_i-\bX_j\bV_j\|_F^2,\] which
has applications to areas such as psychometrics, image and shape analysis and biometric identification.

The non-commutative little Grothendieck problem~\cite{Bandeira2016} is defined by:
\begin{equation}\label{eq:littleGrothendieck}
\min_{\bV_1, \ldots, \bV_n\in O(D)}\sum_{i,j=1}^n\tr(\bC_{ij}\bV_i\bV_j^\top).\end{equation}
Problem \eqref{eq:original00} can be considered as a special case of \eqref{eq:littleGrothendieck} with $n=3$. The argument is as follows. For convenience, we homogenize \eqref{eq:original0} by introducing a slack unitary variable $\bV_3\in O(D)$ and consider the
augmented linear system
\begin{equation}\bX_1\bV_1+\bX_2\bV_2+\bX_3\bV_3=\mathbf{0}\label{eq:original}\end{equation}
Clearly, if $(\bV_1, \bV_2, \bV_3)$ is a solution to \eqref{eq:original}, then the pair $(-\bV_1\bV_3^\top, -\bV_2\bV_3^\top)$ is a solution to the original linear system \eqref{eq:original0}. The least squares formulation corresponding to \eqref{eq:original} is
\begin{equation}\label{eq:original1}
\min_{\bV_1, \bV_2, \bV_3\in O(D)} \|\bX_1\bV_1+\bX_2\bV_2+\bX_3\bV_3\|_F^2.
\end{equation}
Let $\bC\in \reals^{3D\times 3D}$ be a Hermitian matrix with the $(i,j)-$th $D\times D$ block given by $\bC_{ij}=\bX_i^\top\bX_j$. The least squares problem \eqref{eq:original1} is equivalent to
\begin{equation}\label{eq:original15}
\min_{\bV_1, \bV_2, \bV_3\in O(D)} \sum_{i,j=1}^3 \tr(\bC_{ij}\bV_j\bV_i^\top),
\end{equation}
which is the little Grothendieck problem~\eqref{eq:littleGrothendieck} with $n=3$.

\subsection{Motivation}
Our problem arises naturally in single particle reconstruction (SPR) from cryo-electron microscopy (EM), where the goal is to determine the 3D structure of a macromolecule complex from  2D projection
images of identical, but randomly oriented, copies of the macromolecule. Zvi Kam~\cite{kam1980} showed that the spherical harmonic expansion coefficients of the Fourier transform of the 3D molecule, when arranged as matrices, can be estimated from 2D projection images up to an orthogonal matrix (for each degree of spherical harmonics). Based on this observation, Bhamre et al.~\cite{Tejal2014} recently proposed ``Orthogonal Replacement'' (OR), an orthogonal matrix retrieval procedure in which cryo-EM projection images are available for two unknown structures $\phi^{(1)}$ and $\phi^{(2)}$ whose difference $\phi^{(2)} - \phi^{(1)}$ is known. It follows from Kam's theory that we are given the spherical harmonic expansion  coefficients of $\phi^{(1)}$ and $\phi^{(2)}$ up to an orthogonal matrix, and their difference. Then the problem of recovering the spherical harmonic expansion coefficients of $\phi^{(1)}$ and $\phi^{(2)}$ is reduced to the mathematical problem \eqref{eq:original0}. If \eqref{eq:original0} can be solved for smaller $N$, then we can reconstruct $\phi^{(1)}$ and $\phi^{(2)}$ with higher resolution. The cryo-EM application serves as the main motivation of this paper. We refer the reader to \cite{Tejal2014} for further details regarding the specific application to cryo-EM.

%
%
%
%
%
%

\section{Algorithm and Main result}\label{sec:main}
The little Grothendieck problem and QO-OCs are generally intractable, for example, it is well-known that the Max-Cut Problem is NP-hard. Many approximation algorithms have been proposed and analyzed \cite{Goemans1995,Nemirovski2007,So2007,So2011,Afonso2013,Naor2013}, and the principle of these algorithms is to apply a semi-definite programming (SDP) relaxation followed by a rounding procedure.  The SDP can be solved in polynomial time (for any finite precision).  Based on the same principle, we relax the problem \eqref{eq:original15} to an SDP as follows.

Let $\bH\in \reals^{3D\times 3D}$ be a Hermitian matrix with the $(i,j)-$th $D\times D$ block given by $\bH_{ij}=\bV_i\bV_j^\top$, that is,
\[
\bH=\left( \begin{array}{c}
\bV_1 \\
\bV_2 \\
 \bV_3 \end{array} \right)\left(
\bV_1^\top,
\bV_2^\top,
 \bV_3^\top \right).
\] Then \eqref{eq:original15} is equivalent to \[\min_{\bH\psdge\mathbf{0}, \bH_{ii}=\bI, \mathrm{rank}(\bH)=D}\tr(\bC\bH),\] where $\bH\psdge\mathbf{0}$ denotes that $\bH$ is a positive semidefinite matrix. The only constraint which is
non-convex is the rank constraint. Dropping it leads to the following SDP:
\begin{equation}\label{eq:original2}
\text{min $\tr(\bC\bH)$, subject to $\bH\psdge \mathbf{0}$ and $\bH_{ii}=\bI$.}
\end{equation}
If the solution satisfies $\rank(\bH)=D$, then $\bV_1$, $\bV_2$ and $\bV_3$ are extracted by applying decomposition to $\bH$ as follows. Let $\bH=\bU\bU^\top$, where
\[
\bU=\left( \begin{array}{c}
\bU_1 \\
\bU_2 \\
 \bU_3 \end{array} \right)\in\reals^{3D\times D},\,\,\,\text{and $\bU_i\in\reals^{D\times D}$}
\]
then $\bV_i=\bU_i$, $1\leq i\leq 3$ would be a solution.

Notice that the solution to \eqref{eq:original} is not unique: if $(\bU_1, \bU_2, \bU_3)$ satisfies \eqref{eq:original}, then for any $\bU\in O(D)$, the triplet $(\bU_1\bU, \bU_2\bU, \bU_3\bU)$ satisfies \eqref{eq:original} as well. Although the solution to \eqref{eq:original} is not unique, the solution to the original problem~\eqref{eq:original00} is uniquely given by $(-\bU_1\bU_3^\top, -\bU_2\bU_3^\top)$.

When $\rank(\bH)>D$, then there does not exist $\bU\in\reals^{3D\times D}$ such that $\bH=\bU\bU^\top$ and the linear system~\eqref{eq:original0} does not have a solution. However, we could employ the rounding procedure described in~\cite{Afonso2013}: Let $\bH=\bU\bU^\top$, where
\[
\bU=\left( \begin{array}{c}
\bU_1 \\
\bU_2 \\
 \bU_3 \end{array} \right)\in\reals^{3D\times 3D},\,\,\,\text{and $\bU_i\in\reals^{D\times 3D}$},
\]
then we generate approximate solutions by $\bV_1=f(-\bU_1\bU_3^\top)$ and $\bV_2=f(-\bU_2\bU_3^\top)$, where $f$ is a rounding procedure to the nearest orthogonal matrix as follows. For any $\bZ\in\reals^{D\times D}$ with SVD decomposition $\bZ=\bU_{\bZ}\Sigma_{\bZ}\bV_{\bZ}^\top$, $f(\bZ)=\bU_{\bZ}\bV_{\bZ}^\top =\bZ(\bZ^\top\bZ)^{-\frac{1}{2}}$~\cite{Keller1975}. 


\subsection{Main results}

The main contribution of this paper is a particular theoretical guarantee for the SDP approach to return a solution of rank $D$ and recover $\bV_1$ and $\bV_2$ exactly. We start with a theorem that controls the lower bound of the objective function in \eqref{eq:original2}. Throughout the paper, for any $d$-dimensional subspace $\rmL$ in $\reals^D$, $\bP_{\rmL}$ is a projector of size $D\times d$ to the subspace. 

\begin{thm}\label{thm:stability}
For generic $\bX_1,\bX_2\in\reals^{N\times D}$ with $N\geq D+1$, $\bX_3=-\bX_1-\bX_2$, $k\geq D$,  and
\[
\mathcal{U}=\left\{\bU=\left( \begin{array}{c}
\bU_1 \\
\bU_2 \\
 \bU_3 \end{array} \right)\in\reals^{3D\times k}: \bU_i\in\reals^{D\times k},  \bU_1\bU_1^\top=\bU_2\bU_2^\top=\bU_3\bU_3^\top=\bI,\right\},
\]
then for any $\bU\in\mathcal{U}$,
\begin{equation}\label{eq:stability}
\|\bX\bU\|_F\geq c(\bX_1, \bX_2)\|\bU^\top\bP_{\rmL_1^\perp}\|_F^2,
\end{equation}
where $\bX=(\bX_1,\bX_2,\bX_3)\in\reals^{N\times 3D}$, \[
\rmL_1=\{\bx\in\mathbb{R}^{3D}: \bx=(\bv,\bv,\bv)\,\,\,\text{for some $\bv\in\mathbb{R}^{D}$}\},
\]
and $c(\bX_1, \bX_2)$ is a constant depending on $\bX_1$ and $\bX_2$ and it is positive for generic $\bX_1$ and $\bX_2$.
\end{thm}


Based on Theorem~\ref{thm:stability} and $\|\bX\bU\|_F^2=\tr(\bC\bH)$, this paper proves that when $N\geq D+1$, the SDP method recovers the orthogonal matrices for generic cases, i.e., the property holds for $(\bX_1,\bX_2)$ that lies in a dense open subset of $\mathbb{R}^{N\times D}\times \mathbb{R}^{N\times D}$. This is formally stated next. Its part (b) shows that the SDP method is stable to noise.
\begin{thm}\label{thm:probablistic}
(a) For generic $\bX_1,\bX_2\in\reals^{N\times D}$ with $N\geq D+1$, the SDP method recovers $\bV_1$ and $\bV_2$ exactly. \\
(b) Under the assumptions in (a), and suppose that the input matrices of the SDP method are $\hat{\bX_i}$ such that $\|\hat{\bX}_i-\bX_i\|_F\leq \epsilon$ for $1\leq i\leq 3$, then the SDP method recovers $\{\bV_i\}_{i=1}^2$ approximately in the sense that the error between the recovered orthogonal matrix $\hat{\bV}_i$ and the true orthogonal matrix $\bV_i$, $\|\hat{\bV}_i-\bV_{i}\|_F$, is bounded above by $C\sqrt{\epsilon}$ for some $C$ that does not depend on $\epsilon$.
\end{thm}

The result (a) shows that the SDP method successfully recovers the orthogonal matrices as long as $N\geq D+1$, compared with the stringent requirement $N \geq 2D$ for the na\"{i}ve least squares approach. 
The condition $N\geq D+1$ is nearly optimal. In \eqref{eq:original0}, there are $ND$ constraints and $D(D-1)$ variables. Hence, it is impossible to recover $\bV_1$ and $\bV_2$ when $N<D-1$.

The result (b) shows that the SDP method is stable to noise in the input matrices. We remark that it might be possible to improve the stability analysis: While the current analysis gives an error of $O(\sqrt{\epsilon})$, the empirical performance usually has an error of $O(\epsilon)$, as shown in Table~\ref{tab:noise} of Section~\ref{sec:simu}. 

We also remark that Theorem~\ref{thm:probablistic} can be generalized to the complex case---the proof applies to the case of unitary matrices as well. For the complex case, there are $2ND$ constraints and $2D^2$ degrees of freedom. Therefore, it is impossible to recover $\bV_1$ and $\bV_2$ when $N<D$. Moreover, we suspect that recovery is impossible even for $N=D$, which would suggest that the sufficient condition $N\geq D+1$ in Theorem~\ref{thm:probablistic}(a) is also necessary: in fact, it is easy to verify the impossibility of recovering $\bV_1$ and $\bV_2$ when $N=D=1$.
\subsection{Generalization}
A natural generalization of \eqref{eq:original0} is the following problem: given known matrices $\bX_1, \bX_2, \ldots, \bX_{K-1}\in\reals^{N\times D}$ and unknown orthogonal matrices $\bV_1, \bV_2, \ldots, \bV_{K-1}\in O(D)$, recover $\{\bV_i\}_{i=1}^{K-1}$ from  \begin{equation}\label{eq:original0_generalize}\bX_K=\sum_{i=1}^{K-1}\bX_i\bV_i.\end{equation}

For this generalized problem, the SDP method is formulated as follows. We first homogenize it to
\[
\sum_{i=1}^{K}\bX_i\bV_i=\mathbf{0},
\]
and let $\bH\in \reals^{KD\times KD}$ be a Hermitian matrix with the $(i,j)-$th $D\times D$ block given by $\bH_{ij}=\bV_i\bV_j^\top$. Then the SDP method solves
\begin{equation}\label{eq:original3}\text{min $\tr(\bC\bH)$, subject to $\bH\psdge \mathbf{0}$ and $\bH_{ii}=\bI$ for all $1\leq i\leq K$,}
\end{equation}
where $\bH_{ii}$ represents the $(i,i)-$th $D\times D$ block. Then we extract the orthogonal matrices by the procedure described in Section~\ref{sec:main}.

For this generalized problem and its associated SDP approach, we have the following theoretical guarantee.
\begin{thm}\label{thm:probablistic_generalize}
For generic $\{\bX_i\}_{i=1}^{K-1}\in\mathbb{R}^{N\times D}$, if $N\geq (K-2)D+1$,  then the SDP method recovers $\{\bV_i\}_{i=1}^{K-1}$ exactly.
\end{thm}
Theorem~\ref{thm:probablistic}(a) can be considered as a special case of Theorem~\ref{thm:probablistic_generalize} when $K=3$. However, for $K>3$, the condition $N\geq (K-2)D+1$ is not close to optimal. Since \eqref{eq:original0_generalize} has $ND$ constraints and $D(D-1)(K-1)/2$ variables, the information-theoretic limit is $N=(D-1)(K-1)/2$. Simulations in Section~\ref{sec:simu} also show that the SDP approach empirically recovers the orthogonal matrices even when $N$ is smaller than $(K-2)D+1$. However, the theoretical guarantee in Theorem~\ref{thm:probablistic_generalize} is still more powerful than the least squares approach of solving $\min_{\{\bV_i\}_{i=1}^{K-1}\in \reals^{D\times D}}\|\sum_{i=1}^{K-1}\bX_i\bV_i-\bX_K\|_F^2$, which requires $N\geq (K-1)D+1$ to recover $\{\bV_i\}_{i=1}^{K-1}$.

\section{Numerical Experiments}\label{sec:simu}
In this section, we compare several methods for solving \eqref{eq:original0} and \eqref{eq:original0_generalize} on artificial data sets. The data sets are generated as follows: $\{\bX_i\}_{i=1}^{K-1}$ are random matrices with i.i.d standard Gaussian entries $\mathcal{N}(0,1)$, $\{\bV_i\}_{i=1}^{K-1}$ are random orthogonal matrices  (according to Haar measure) generated by QR decomposition of random matrices with i.i.d standard Gaussian entries, and $\bX_K$ is generated by  \eqref{eq:original0_generalize}.

We compare the following five methods:
\begin{enumerate}
\item The SDP relaxation approach (SDP) described in Section~\ref{sec:main}.
\item The na\"{i}ve least squares approach (LS):
\[\min_{\{\bV_i\}_{i=1}^{K-1}}\|\bX_K-\sum_{i=1}^{K-1}\bX_i\bV_i\|_F^2\]
\item Since the convex hull of the set of orthogonal matrices is the set of matrices with operator norm not greater than one, we can strengthen the LS approach by constraining its domain (C-LS):
\[\min_{\{\bV_i\}_{i=1}^{K-1}}\|\bX_K-\sum_{i=1}^{K-1}\bX_i\bV_i\|_F^2,\,\,\text{subject to $\|\bV_i\|\leq 1$, $1\leq i\leq K-1$}\]
\item This is an approach suggested to us by Afonso Bandeira. Let us start with the case $K=3$. If $\bV_3=\bV_1\bV_2^\top$, then from \eqref{eq:original0},
$\bX_3\bV_2^\top=\bX_1\bV_3+\bX_2$ and $\bX_3\bV_1^\top=\bX_1+\bX_2\bV_3^\top$. Then we solve the expanded least squares problem based on these three equations (LS+):
\[\min_{\bV_1,\bV_2,\bV_3}\|\bX_3-\bX_1\bV_1-\bX_2\bV_2\|_F^2+
\|\bX_3\bV_2^\top-\bX_1\bV_3-\bX_2\|_F^2+\|
\bX_3\bV_1^\top-\bX_1-\bX_2\bV_3^\top\|_F^2.
\]
Defining
\[
\bH=\left( \begin{array}{ccc}
\bI & \bV_3  & -\bV_1 \\
\bV_3^\top & \bI & -\bV_2 \\
 -\bV_1^\top & -\bV_2^\top  & \bI \end{array} \right),
\]
the optimization problem can be rewritten as
\[
\min_{\bH\in\mathbb{R}^{3D\times 3D}}\|(\bX_1,\bX_2,\bX_3)\bH\|_F^2, \,\,\,\text{subject to $\bH=\bH^\top$, $\bH_{ii}=\bI$}.
\]
In general, for $K\geq 3$, this method can be formulated as \[\text{$\min_{\bH\in\mathbb{R}^{KD\times KD}}{\tr(\bC\bH^2)}$, subject to $\bH=\bH^\top$ and $\bH_{ii}=\bI$ for all $1\leq i\leq K$,}\]
where $\bH_{ij}$ represents the $ij$-th $D\times D$ block of $\bH$.
\item The LS+ approach with constraints on the operator norm of $\bH_{ij}$ (C-LS+):
\[\text{$\min_{\bH\in\mathbb{R}^{KD\times KD}}{\tr(\bC\bH^2)}$, subject to $\bH=\bH^\top$, $\bH_{ii}=\bI$ and $\|\bH_{ij}\|\leq 1$ for all $1\leq i,j\leq K$.}\]

\end{enumerate}

To compare the SDP/LS+/C-LS+ approaches, we summarize their objective functions and their constraints in  Table~\ref{tab:compare}. There are two main differences. First, the objective functions are different. However, since $\tr(\bC\bH)=0$ if and only if $\tr(\bC\bH^2)=0$ (considering $\bC\psdge \mathbf{0}$ and $\bH\psdge \mathbf{0}$), this difference does not affect the property of exact recovery. Second, the constraints of the SDP approach are more restrictive than those of the C-LS+ approach ($\bH_{ii}=\bH_{jj}=\bI$ and $\bH\psdge \mathbf{0}$ imply $\|\bH_{ij}\|\leq 1$), which is more restrictive than the C-LS approach. This observation partially justifies the fact that SDP performs better than C-LS+, and C-LS+ performs better than C-LS. However, these interpretations do not justify the empirical finding in Figures~\ref{fig:DD} and~\ref{fig:generalize} that C-LS+ and SDP behave very similarly in the absence of noise. We leave the explanation of this observation as an open question.

\begin{table*}[htbp]
\centering
\caption{\small {Comparison between SDP, LS+ and C-LS+ approaches. \label{tab:compare}}}
\begin{tabular}{|c|c|c|c|}
\hline
 & objective function  & common constraint & other constraints \\
 \hline
SDP&$\tr(\bC\bH)$& &$\bH\psdge\mathbf{0}$\\
LS+&$\tr(\bC\bH^2)$& $\bH_{ii}=\bI$, $\bH=\bH^\top$&\\
C-LS+&$\tr(\bC\bH^2)$& &$\|\bH_{ij}\|\leq 1$\\
\hline
\end{tabular}
\end{table*}

Among these optimization approaches, the LS method has an explicit solution by decomposing it into $D$ sub-problems, where each is a regression
problem that estimates $KD$ regression parameters. 
All other methods are convex and can be solved by CVX~\cite{cvx}, where the default solver SeDuMi is used~\cite{SDPT2003}. While the LS+ approach can also be written as a least squares problem with an explicit solution, this problem is not decomposable (unlike the LS method).

When the solution matrices of LS/C-LS are not orthogonal, they are rounded to the nearest orthogonal matrices using the approach in~\cite{Keller1975}. The rounding procedure of LS+/C-LS+ is the same as that of the SDP method.

In the first simulation, we aim to find the size of $N$ such that the orthogonal matrices be exactly recovered by the suggested algorithms for $K=3$. We let $D=10$ or $20$ and choose various values for $N$, and record the mean recovery error of $\bV_1$ (in Frobenius norm) over 50 repeated simulations in Figure~\ref{fig:DD}. The performance
of LS verifies our theoretical analysis: it recovers the orthogonal matrices for $N\geq 2D$. LS fails when $N<2D$ because the null space of $[\bX_1,\bX_2]$ is nontrivial and there are infinite solutions. Besides, LS+ succeeds when  $N\geq 3D/2$. SDP and C-LS+ are the best approaches and they succeed when $N\geq D+1$, which verifies Theorem~\ref{thm:probablistic}. 

In the second simulation, we test the stability of the suggested algorithms when $K=3$ and the measurement matrix $\bX_3$ is contaminated elementwisely by Gaussian noise $\mathcal{N}(0,\sigma^2)$. We use the setting $N=12,16,22$, $D=10$ and $\sigma=0.01$ or $0.1$ and record the mean recovery error over 50 runs in Table~\ref{tab:noise}, which shows that the SDP relaxation approach is more stable to noise than competing approaches. This motivates our interest in studying the SDP approach.

\begin{figure}
\centering\includegraphics[width=0.48\textwidth]{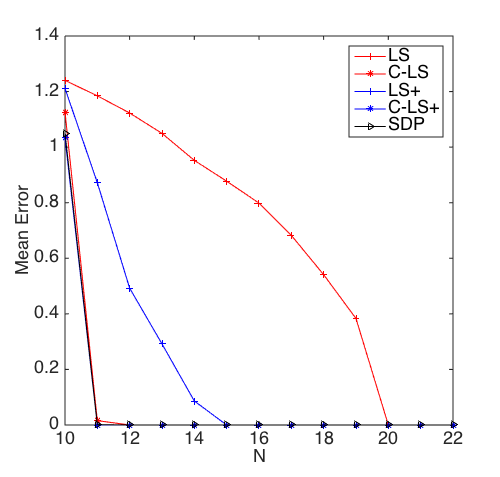}\includegraphics[width=0.48\textwidth]{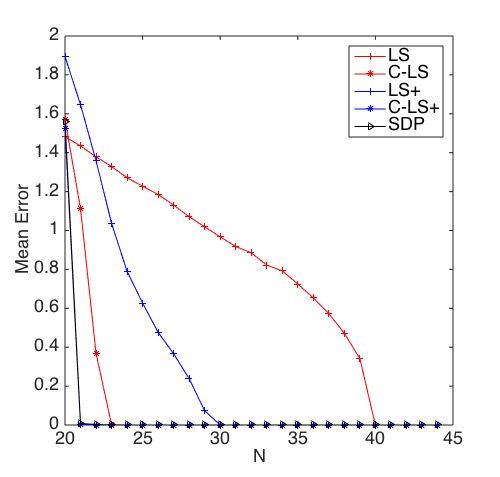}
\caption{The dependence of the mean recovery error (over $50$ runs) with respect to $N$, when $D=10$ (left panel) and $D=20$ (right panel). The $y$-axis represents the mean recovery error of $\bV_1$ in Frobenius norm.}\label{fig:DD}
\end{figure}

\begin{table*}[htbp]
\centering
\caption{\small {The mean recovery error over $50$ runs in the noisy setting for $K=3$ and $D=10$. \label{tab:noise}}}
\begin{tabular}{|c|c|c|c|c|c|c|}
\hline
$N$ & $\sigma$  & SDP&C-LS+ & LS+ & C-LS& LS \\
\hline
12& 0.01 &   \textbf{0.071}  &  0.076 &   0.482  &  0.508&    2.260 \\
16& 0.01 &    \textbf{0.026}  &  0.031    & 0.037 &   0.059  &  1.926 \\
22& 0.01 &    \textbf{0.018}  &  0.021  &  0.020  &  0.030 &   0.077 \\
\hline
12& 0.1 &    \textbf{0.742}  &  \textbf{0.742}  &  1.088  &  0.880  &  2.341 \\
16& 0.1 &   \textbf{ 0.261}  &  0.328  &  0.399  &  0.459  &  2.034 \\
22& 0.1 &    \textbf{0.175}  &  0.217  &  0.216  &  0.262  &  0.834 \\
\hline
\end{tabular}
\end{table*}




In the third simulation, we compare these methods for $K=5$ and $D=5,10$. The results are shown in Figure~\ref{fig:generalize}. This simulation verifies Theorem~\ref{thm:probablistic_generalize} by showing that the SDP approach successfully recovers the orthogonal matrices for $N\geq (K-2)D+1$. Indeed, the empirical performance of the SDP approach is even better: it recovers $\{\bV_i\}_{i=1}^5$ at $N=12$ and $25$ respectively, which are smaller than $(K-2)D+1$. Compared with LS/LS+/C-LS, the SDP and C-LS+ approaches recover the orthogonal matrices with smaller $N$.
\begin{figure}\label{fig:generalize}
\centering\includegraphics[width=0.48\textwidth]{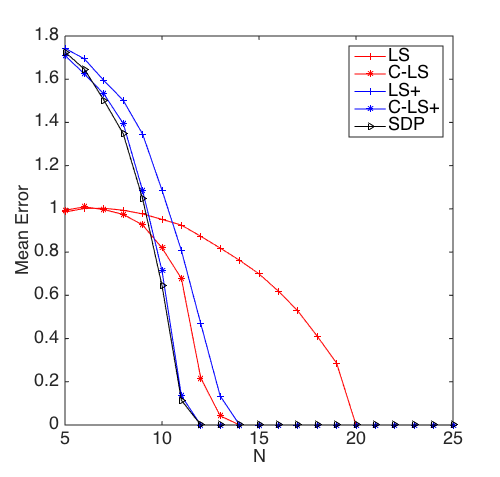}\includegraphics[width=0.48\textwidth]{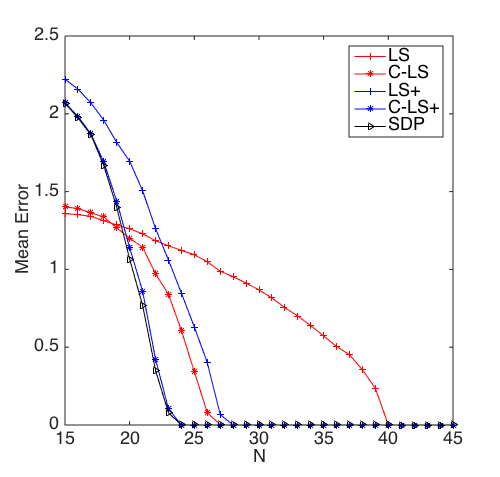}
\caption{The dependence of the mean recovery error (over $50$ runs) with respect to $N$, when $K=5$, $D=5$ (left panel) and $K=5$, $D=10$ (right panel). The $y$-axis represents the mean recovery error of $\bV_1$ in Frobenius norm.}
\end{figure}

At last, we record the running time for all  approaches in Table~\ref{tab:time}. Although the running time is not the main focus of this paper, and CVX is not optimized for the approaches, this table gives  a sense of the running times. Table~\ref{tab:time} clearly shows that the LS approach is much faster than the other approaches, and the SDP approach is consistently faster than C-LS+. We suspect that it is due to the fact that SDP has fewer constraints, even though the constraint of SDP is more restrictive than that of C-LS+.

\begin{table*}[htbp]
\centering
\caption{\small {The average running time (in seconds) when $\sigma=0.1$, $K=3$, $D=10$, $N=15$ (first row) and $\sigma=0.1$, $K=5$, $D=10$, $N=35$ (second row). \label{tab:time}}}
\begin{tabular}{|c|c|c|c|c|}
\hline
 SDP&C-LS+ & LS+ & C-LS& LS \\
\hline
0.64  &  0.76   &  0.20 &  0.48 & 0.0005   \\
5.33  &  6.58   &  0.62  & 0.93 & 0.0017\\
\hline
\end{tabular}
\end{table*}

\section{Proofs of main results}\label{sec:proof}
In this section, we first provide the proof for Theorem~\ref{thm:probablistic}, assuming Theorem~\ref{thm:stability}, and then  provide the proofs for Theorems~\ref{thm:stability} and~\ref{thm:probablistic_generalize}. The main reason for this organization is that, given Theorem~\ref{thm:stability} (whose proof is more technical), the proof of Theorem~\ref{thm:probablistic} is rather straightforward. This organization would also emphasize the importance of Theorem~\ref{thm:stability}, which plays an important role in the proof of Theorems~\ref{thm:probablistic}.
\subsection{Proof of Theorem~\ref{thm:probablistic}}
Part (a) follows from the result in part (b) with $\epsilon = 0$, so it is sufficient to prove part (b).

In the proof of part (b), we first claim that it is sufficient to prove the case $\bV_1=\bV_2=-\bI$, i.e.,  when $\bX_1+\bX_2+\bX_3=\mathbf{0}$ and $\|\hat{\bX}_i-\bX_i\|_F\leq \epsilon$ for $1\leq i\leq 3$, then the SDP method recovers $\hat{\bV}_1$ and $\hat{\bV}_2$ such that $\|\hat{\bV}_i+\bI\|\leq C\sqrt{\epsilon}$ for $i=1,2$.

This result implies that if the input of the SDP method is $(-\hat{\bX}_1\bV_1, -\hat{\bX}_2\bV_2, \hat{\bX}_3)$ and the output is denoted by $(\tilde{\bV}_1, \tilde{\bV}_2)$, then $\|\tilde{\bV}_i+\bI\|_F\leq C\sqrt{\epsilon}$ for $i=1,2$.

Additionally, it can be verified that if the SDP method outputs $(\bV_1,\bV_2)$ when the input is $(\bX_1, \bX_2, \bX_3)$, $\bU_1$ and $\bU_2$ are two orthogonal matrices, then the output for $(\bX_1\bU_1, \bX_2\bU_2, \bX_3)$ would be $(\bU_1^\top\bV_1, \bU_2^\top\bV_2)$. Applying this observation, we have $\tilde{\bV}_i=-\hat{\bV}_i\bV_i^\top$. Combining it with  $\|\tilde{\bV}_i+\bI\|_F\leq C\sqrt{\epsilon}$ for $i=1,2$, the theorem is proved.



The rest of the proof will assume $\bV_1=\bV_2=-\bI$ and $\bX_1+\bX_2+\bX_3=\mathbf{0}$. We represent the noisy setting in \eqref{eq:original2} by $\hat{\bX}$, $\hat{\bC}$ and $\hat{\bH}$, the clean setting by $\bX$, $\bC$ and $\bH$, and write the decomposition of $\bH$ and $\hat{\bH}$ by $\bH=\bU\bU^\top$ and $\hat{\bH}=\hat{\bU}\hat{\bU}^\top$.

Since $\hat{\bU}_i\hat{\bU}_i^\top=\bI$, $\|\hat{\bU}\|\leq\sum_{i=1}^3\|\hat{\bU}_i\|=3$, \eqref{eq:frobenius} implies
\begin{equation}\label{eq:stability1}
\|\bX\hat{\bU}\|_F-\|\hat{\bX}\hat{\bU}\|_F\leq \|(\bX-\hat{\bX})\hat{\bU}\|_F
\leq 3\|\bX-\hat{\bX}\|_F\leq 9\epsilon
\end{equation}
and following the same argument,
\begin{equation}\label{eq:stability3}
\|\hat{\bX}\bU\|_F\leq 9\epsilon+\|{\bX}\bU\|_F.
\end{equation}
Since $\hat{\bH}=\hat{\bU}\hat{\bU}^\top$ is the minimizer of the SDP problem, we have $\|\hat{\bX}\hat{\bU}\|_F\leq \|\hat{\bX}\bU\|_F$. Combining it with $\|{\bX}\bU\|_F=0$ and \eqref{eq:stability3},
\begin{equation}\label{eq:stability0}
\|\hat{\bX}\hat{\bU}\|_F\leq 9\epsilon.
\end{equation}
Combining \eqref{eq:stability1}, \eqref{eq:stability0} with Theorem~\ref{thm:stability},
\begin{equation}\label{eq:stability2}
9\epsilon\geq \|\hat{\bX}\hat{\bU}\|_F\geq \|{\bX}\hat{\bU}\|_F-9\epsilon
\geq
c(\bX)\|\hat{\bU}^\top\bP_{\rmL_1^\perp}\|_F^2-9\epsilon,\end{equation}
so we have
\begin{equation}\label{eq:stability4}
\|\hat{\bU}^\top\bP_{\rmL_1^\perp}\|_F\leq \sqrt{\frac{18\epsilon}{c(\bX)}}.
\end{equation}
Combining it with Lemma~\ref{lemma:frobenius},
\begin{align*}
\|\hat{\bU}_1\hat{\bU}_3^\top-\bI\|_F=&\|\hat{\bU}_1-\hat{\bU}_3\|_F=
\|\hat{\bU}^\top(\bI,\mathbf{0},-\bI)\|_F
=
\|\hat{\bU}^\top\bP_{\rmL_1^\perp}\bP_{\rmL_1^\perp}^\top(\bI,\mathbf{0},-\bI)\|_F\\
\leq &\|\hat{\bU}^\top\bP_{\rmL_1^\perp}\|_F \|\bP_{\rmL_1^\perp}^\top(\bI,\mathbf{0},-\bI)\|
\leq 2\sqrt{\frac{18\epsilon}{c(\bX)}}.
\end{align*}

Since the post-processing step $f(\bZ)$ is a continuous and differentiable function with respect to $\bZ$ and $f(-\bI)=-\bI$, the difference between the recovered orthogonal matrix $\hat{\bV}_1$ and $\bV_1=-\bI$, $\|\hat{\bV}_1+\bI\|_F$, is bounded above by $C\sqrt{\epsilon}$ when $c(\bX)> 0$. Similarly we have $\|\hat{\bV}_2+\bI\|_F<C\sqrt{\epsilon}$.


\subsection{Proof of Theorem~\ref{thm:stability}}
\begin{proof}
The main idea of the proof is to investigate $\tilde{\bU}^*$, which is defined to be nearest matrix to $\bU$ in the set $\tilde{\mathcal{U}}$ defined in \eqref{eq:tildeU}. Then we represent $\tilde{\bU}^*$ in the form of \eqref{eq:representation}, and show that $\|\bL\|_F$ is bounded by $C_1 \|\bX\bU\|_F$, for some $C_1>0$. With additional bounds $\|\bK\bY_1\|_F, \|\bK\bY_2\|_F\leq C_2\sqrt{ \|\bX\bU\|_F}$ and \[
\|\tilde{\bU}^{*\top}\bP_{\rmL_1^\perp}\|_F\leq \left\|\left( \begin{array}{ccc}
(\bL\bY_1)^\top & (\bL\bY_2)^\top & \mathbf{0} \\
(\bK\bY_1)^\top & (\bK\bY_2)^\top & \mathbf{0} \end{array}\right)\right\|_F,
\]
we will show that $\|\tilde{\bU}^{*\top}\bP_{\rmL_1^\perp}\|_F$ is bounded above by a function of $\|\bX\bU\|_F$ in \eqref{eq:key}. By analyzing  the properties of $\tilde{\bU}^*$, the same statement holds for $\|\bU\bP_{\rmL_1^\perp}\|_F$, and the theorem is proved.

We first remark that it is sufficient to prove the case $N=D+1$. If this is true, then for $N>D+1$, \eqref{eq:stability} holds when $\bX$ is replaced by $\bX'$, the submatrix consists of the first $D+1$ rows of $\bX$. Since $\|\bX\bU\|_F\geq \|\bX'\bU\|_F$, \eqref{eq:stability} is proved. Therefore, for the rest of the proof, we assume $N=D+1$.

We first define the following:\begin{align}\nonumber
\bar{\mathcal{U}}&=\{\bar{\bU}\in\reals^{3D\times k}: \bX\bar{\bU}=\mathbf{0}\},
\\
\tilde{\mathcal{U}}&=\{\tilde{\bU}\in\reals^{3D\times k}: \bX\tilde{\bU}=\mathbf{0}, \tilde{\bU}_3=\bU_3\},\label{eq:tildeU}
\end{align}
and in \eqref{eq:tildeU}, $\tilde{\bU}_3$ and $\bU_3$ represent the submatrix consists of the last $D$ rows of $\tilde{\bU}$ and $\bU$. We also define the distances between two matrices and the distances between a matrix and a set by
\[
\dist(\bU,\bU')=\|\bU-\bU'\|_F,\,\,\,\,\dist(\bU,\mathcal{U})=\min_{\bU'\in\mathcal{U}}\|\bU-\bU'\|_F.
\]
Then we have
\begin{equation}\label{eq:bartilde}
\dist(\bU,\bar{\mathcal{U}})\geq C\dist(\bU,\tilde{\mathcal{U}}),
\end{equation}
for $C=1/2$. The proof of \eqref{eq:bartilde} is deferred to Section~\ref{sec:bartilde}.

Assuming that $\sigma_{\min}(\bX)$ is the smallest singular value of $\bX$, and for any matrix $\bA\in\mathbb{R}^{m\times n}$, $\Sp(\bA)$ represents the subspace spanned by the row vectors of $\bA$ in $\reals^n$, then we have
\begin{align}\label{eq:dist_constrol1}
&\|\bX\bU\|_F\geq \sigma_{\min}(\bX)\|\bP_{\Sp(\bX)}^\top\bU\|_F = \sigma_{\min}(\bX) \dist(\bU,\bar{\mathcal{U}})\geq C\sigma_{\min}(\bX) \dist(\bU,\tilde{\mathcal{U}}),
\end{align}
where the last inequality follows from \eqref{eq:bartilde}.

Assuming that $\tilde{\bU}^*=\argmin_{\tilde{\bU}\in\tilde{\mathcal{U}}}\dist(\tilde{\bU},\bU)$, then using $\left(\begin{array}{c}
\bU_3 \\
\bU_3 \\
 \bU_3 \end{array} \right)\in \tilde{\mathcal{U}}$ we have
\begin{equation}\label{eq:result1}
\dist(\tilde{\bU}^*,\bU)\leq \dist\left(\left(\begin{array}{c}
\bU_3 \\
\bU_3 \\
 \bU_3 \end{array} \right),\bU\right)\leq \sqrt{8D}.
\end{equation}

Now let us investigate $\dist(\tilde{\bU}^*,\mathcal{U})$ further. If $\bY_1$ and $\bY_2\in\reals^{(2D-N)\times D}$ are chosen such that $[\bX_1,\bX_2][\bY_1,\bY_2]^\top=\mathbf{0}$, then using $\bX_1+\bX_2+\bX_3=\mathbf{0}$, there exist $\bL\in\reals^{D\times (3D-N)}$, $\bK\in\reals^{(k-D)\times (3D-N)}$ and an orthogonal matrix $\bU_3'\in\reals^{k\times k}$ such that
\begin{equation}
\tilde{\bU}^*=\left( \begin{array}{ccc}
\bI+\bL\bY_1 & \bK\bY_1\\
\bI+\bL\bY_2 &  \bK\bY_2  \\
 \bI& \mathbf{0} \end{array} \right)\bU_3'.\label{eq:representation}
\end{equation}
 That is, if we write $\tilde{\bU}^*=\left(\begin{array}{c}
\tilde{\bU}^*_1 \\
\tilde{\bU}^*_2 \\
 \tilde{\bU}^*_3 \end{array} \right)$, then
\[
\tilde{\bU}^*_1=\left( \bI+\bL\bY_1, \bK\bY_1 \right)\bU_3',\,\,\,
\tilde{\bU}^*_2=\left( \bI+\bL\bY_2, \bK\bY_2 \right)\bU_3',\,\,\,
\tilde{\bU}^*_3=\bU_3
\]
($\tilde{\bU}^*_3=\bU_3$ follows from definition of $\tilde{\mathcal{U}}$). Since for any $\bU_1\in\reals^{k\times D}$ with SVD decomposition $\bU_1=\bU_{\bU_1}\bSigma_{\bU_1}\bV_{\bU_1}^\top$, the closest orthogonal matrix in $\reals^{k\times D}$ is given by $\bU_{\bU_1}\bV_{\bU_1}^\top$, so  the distance between $\tilde{\bU}^*$ and $\mathcal{U}$ is
\[
 \dist(\tilde{\bU}^*,\mathcal{U})=\sqrt{\sum_{i=1}^3\|\bSigma_{\tilde{\bU}^*_i}-\bI\|_F^2}.
\]
Applying \eqref{eq:result1}, all singular values of $\tilde{\bU}_1$, $\tilde{\bU}_2$ and $\tilde{\bU}_3$ (i.e., all diagonal entries of $\{\bSigma_{\tilde{\bU}^*_i}\}_{i=1}^3$) are smaller than $\sqrt{8D}+1$.  Let $C'=\sqrt{8D}+2$, then  $\dist(\bU,\mathcal{U})$
can be controlled as follows:
\begin{align}\nonumber
&C'^2\dist^2(\bU,\tilde{\mathcal{\bU}})=C'^2\dist^2(\tilde{\bU}^*,\bU)\geq C'^2\dist^2(\tilde{\bU}^*,\mathcal{U})
=C'^2\sum_{i=1}^3\|\bSigma_{\tilde{\bU}^*_i}-\bI\|_F^2\\\nonumber\geq& \sum_{i=1}^3\|\bSigma_{\tilde{\bU}^*_i}^2-\bI\|_F^2
={\|\tilde{\bU}^*_1\tilde{\bU}^{*\top}_1-\bI\|_F^2+\|\tilde{\bU}^*_2\tilde{\bU}^{*\top}_2-\bI\|_F^2+\|\tilde{\bU}^*_3\tilde{\bU}^{*\top}_3-\bI\|_F^2}
\\=&\|\bL\bY_1+\bY_1^\top\bL^\top-\bY_1^\top(\bL^\top\bL+\bK^\top\bK)\bY_1\|_F^2\nonumber
\\&+\|\bL\bY_2+\bY_2^\top\bL^\top-\bY_2^\top(\bL^\top\bL+\bK^\top\bK)\bY_2\|_F^2
\label{eq:dist_constrol3}\\\geq &(\min(0,\lambda_{\min}(\bL\bY_1+\bY_1^\top\bL^\top)))^2+
(\min(0,\lambda_{\min}(\bL\bY_2+\bY_2^\top\bL^\top)))^2,\label{eq:dist_constrol2}
\end{align}
where the last inequality follows from the observation that $\bY_1^\top(\bL^\top\bL+\bK^\top\bK)\bY_1$ and $\bY_2^\top(\bL^\top\bL+\bK^\top\bK)\bY_2$ are positive semidefinite, and for any symmetric matrix $\bX$, the distance to the nearest positive semidefinite matrix (in Forbenius norm) is at least $-\min(0,\lambda_{\min}(\bX))$.

Then \eqref{eq:dist_constrol1} and \eqref{eq:dist_constrol2} imply
\[
\lambda_{\min}(\bL\bY_1+\bY_1^\top\bL^\top)\geq - \frac{C'\|\bX\bU\|_F}{C\sigma_{\min}(\bX)},\,\,\,
\lambda_{\min}(\bL\bY_2+\bY_2^\top\bL^\top)\geq - \frac{C'\|\bX\bU\|_F}{C\sigma_{\min}(\bX)}.
\]
Now we introduce an important lemma.
\begin{lemma}\label{lemma:pertubation}For $D>M$,  $\bL\in\reals^{D\times M}$ and $\bY, \bZ\in\reals^{M\times D}$, if $\lambda_{\min}(\bL\bY+\bY^\top\bL^\top)\geq - \epsilon$ and $\lambda_{\min}(\bL\bZ+\bZ^\top\bL^\top)\geq -\epsilon$, where $\lambda_{\min}$ represents the smallest eigenvalue, then $\|\bL\|_F\leq \epsilon c(\bY,\bZ)$, where $c(\bY,\bZ)>0$ for generic $\bY$ and $\bZ$.\end{lemma}
The proof of Lemma~\ref{lemma:pertubation} is rather technical and deferred to Section~\ref{sec:lemmaproof}. Applying Lemma~\ref{lemma:pertubation}, for generic $\bY_1$ and $\bY_2$ (and as a result, for generic $\bX_1$ and $\bX_2$) there exists $C_1$ depending on $\bY_1$ and $\bY_2$ such that $\|\bL\|_F\leq C_1\|\bX\bU\|_F$. In addition, we have
\begin{align*}
&\|\bL\bY_1+\bY_1^\top\bL^\top-\bY_1^\top(\bL^\top\bL+\bK^\top\bK)\bY_1\|_F\\\geq& \|\bY_1^\top\bK^\top\bK\bY_1\|_F-
 \|\bY_1^\top\bL^\top\bL\bY_1\|_F-\|\bL\bY_1+\bY_1^\top\bL^\top\|_F,
\end{align*}
and $ \|\bY_1^\top\bK^\top\bK\bY_1\|_F\geq \frac{1}{\sqrt{D}}\tr(\bY_1^\top\bK^\top\bK\bY_1) =\frac{1}{\sqrt{D}}\|\bK\bY_1\|^2_F$. As a result, there exists $C_2$ depending on $\bY_1, \bY_2, \bX$ such that if $\|\bK\bY_1\|_F>C_2\sqrt{\|\bX\bU\|_F}$ then
\[
\|\bL\bY_1+\bY_1^\top\bL^\top-\bY_1^\top(\bL^\top\bL+\bK^\top\bK)\bY_1\|_F
\geq \frac{C'\|\bX\bU\|_F}{C\sigma_{\min}(\bX)},
\]
which violates \eqref{eq:dist_constrol3} and \eqref{eq:dist_constrol1}.
Therefore, by contradiction we proved $\|\bK\bY_1\|_F\leq C_2\sqrt{\|\bX\bU\|_F}$, and similarly, $\|\bK\bY_2\|_F\leq C_2\sqrt{\|\bX\bU\|_F}$. Combining it with $\|\bL\|_F\leq C_1\|\bX\bU\|_F$,  we have
\begin{align}\nonumber
&\|\tilde{\bU}^{*\top}\bP_{\rmL_1^\perp}\|_F
=\|\tilde{\bU}^{*\top}\bP_{\rmL_1^\perp}\|_F
=\left\|\left( \begin{array}{ccc}
(\bL\bY_1)^\top & (\bL\bY_2)^\top & \mathbf{0} \\
(\bK\bY_1)^\top & (\bK\bY_2)^\top & \mathbf{0} \end{array}\right)\bP_{\rmL_1^\perp}\right\|_F
\\\nonumber\leq& \left\|\left( \begin{array}{ccc}
(\bL\bY_1)^\top & (\bL\bY_2)^\top & \mathbf{0} \\
(\bK\bY_1)^\top & (\bK\bY_2)^\top & \mathbf{0} \end{array}\right)\right\|_F\leq \|\bL\bY_1\|_F+\|\bL\bY_2\|_F+
\|\bK\bY_1\|_F+\|\bK\bY_2\|_F
\nonumber\\\leq& \|\bL\|_F(\|\bY_1\|+\|\bY_2\|)+\|\bK\bY_1\|_F+\|\bK\bY_2\|_F\nonumber\\\leq& (\|\bY_1\|+\|\bY_2\|)C_1\|\bX\bU\|_F+2C_2\sqrt{\|\bX\bU\|_F}.\label{eq:key}
\end{align}
Recall that $C_1$ and $C_2$ only depends on $\bX_1$ and $\bX_2$ ($\bX$, $\bY_1$, and $\bY_2$ are generated from $\bX_1$ and $\bX_2$), combining \eqref{eq:dist_constrol1} and \eqref{eq:key} we have
\begin{align*}
&\|{\bU}^{\top}\bP_{\rmL_1^\perp}\|_F\leq \|\tilde{\bU}^{*\top}\bP_{\rmL_1^\perp}\|_F +\|(\bU-\tilde{\bU}^*)^\top\bP_{\rmL_1^\perp}\|_F\leq \|\tilde{\bU}^{*\top}\bP_{\rmL_1^\perp}\|_F +\|\bU-\tilde{\bU}^*\|_F\\
\leq &(\|\bY_1\|+\|\bY_2\|)C_1\|\bX\bU\|_F+2C_2\sqrt{\|\bX\bU\|_F} + \frac{1}{C\sigma_{\min}(\bX)}\|\bX\bU\|_F.
\end{align*}
Considering that $\|\bU\|\leq 3$ and $\|\bX\bU\|_F\leq \|\bU\|\|\bX\|_F\leq 3\|\bX\|_F$, so if we let $C_4=3\|\bX\|_F$, then
\[
\|{\bU}^{\top}\bP_{\rmL_1^\perp}\|_F\leq\left((\|\bY_1\|+\|\bY_2\|)C_1\sqrt{C_4}+2C_2+\frac{1}{C\sigma_{\min}(\bX)}\sqrt{C_4}\right)\sqrt{\|\bX\bU\|_F},
\]
and Theorem~\ref{thm:stability} is proved.
\end{proof}
\subsubsection{Proof of \eqref{eq:bartilde}}\label{sec:bartilde}
Suppose that $\bar{\bU}^*=\argmin_{\bar{\bU}\in\bar{\mathcal{U}}}\dist(\bar{\bU},\bU)$ and
\[
\bU-\bar{\bU}^*=\left(\begin{array}{c}
\bV_1 \\
\bV_2 \\
 \bV_3 \end{array} \right),\,\,\,\text{where $\bV_i\in\reals^{D\times k}$ for $i=1,2,3$},
\]
then
\[
\tilde{\bU}=\bar{\bU}^*+\left(\begin{array}{c}
 \bV_3\\
 \bV_3\\
 \bV_3 \end{array} \right)
\]
would satisfies that $\tilde{\bU}\in\tilde{\mathcal{U}}$, and as a result,
\begin{align*}
\dist(\bU,\bar{\mathcal{U}})=&\dist(\bU,\bar{\bU}^*)=\sqrt{\|\bV_1\|_F^2+\|\bV_2\|_F^2+\|\bV_3\|_F^2} \geq C\sqrt{\|\bV_1+\bV_3\|_F^2+\|\bV_2+\bV_3\|_F^2}\\=&C \dist(\bU,\tilde{\bU})\geq C\dist(\bU,\tilde{\mathcal{U}}),
\end{align*}
where $C$ can chosen to be $1/2$.

\subsection{Proof of Theorem~\ref{thm:probablistic_generalize}}
\begin{proof}
We start with the same argument as in the proof of Theorem~\ref{thm:probablistic} and assume $\bV_i=-\bI$ for all $1\leq i\leq K-1$ and $\sum_{i=1}^K\bX_i=\mathbf{0}$. Then the proof can be divided into three steps. First, we show that it is sufficient to prove that a property defined in \eqref{eq:perp1} is satisfied for generic  $\bY\in\reals^{p\times KD}$ where $p=\max\left((K-1)D-N,0 \right)$. Then we establish that \eqref{eq:perp1} indeed holds for generic $\bY$. To this end, secondly, we show that any $\bY$ that does not satisfy this property lies in a certain set. Finally, in the third step, we show that this set is of measure zero.


\subsubsection{Step 1: reduction of the problem to the property \eqref{eq:perp1}}
By the assumption that $\bV_i=-\bI$ for all $1\leq i\leq K-1$, $\bX_1+\ldots+\bX_K=\mathbf{0}$, and
\[
\tr\Big(\bH(\bI,\ldots,\bI)^\top(\bI,\ldots,\bI)\Big)=\tr\Big((\bI,\ldots,\bI)\bH(\bI,\ldots,\bI)^\top\Big)=\|\bX_1+\ldots+\bX_K\|_F^2=0.
\]
Considering that $\tr(\bH\bC)\geq 0$ for any $\bH\psdge \mathbf{0}$, if the solution to the SDP problem is not uniquely given by $(\bI,\ldots,\bI)^\top(\bI,\ldots,\bI)$, then there exists $\bH\neq (\bI,\ldots,\bI)^\top(\bI,\ldots,\bI)$ such that $\tr(\bC\bH)=0$. Let $\bH=\bU\bU^\top$ for a matrix $\bU\in\mathbb{R}^{KD\times k}$, then using the properties of $\bH$, we have $\bU\in\mathcal{U}$ for\[\mathcal{U}=\left\{\bU=\left( \begin{array}{c}
\bU_1 \\
\vdots \\
 \bU_K \end{array} \right): \bU_i\in\mathbb{R}^{D\times k}, \bU_i\bU_i^\top=\bI, \forall 1\leq i\leq K, \bP_{\rmL_1^\perp}^\top\bU\neq \mathbf{0}\right\},\]
 $\rmL_1$ defined  by 
\[
\rmL_1=\{\bz\in\mathbb{R}^{KD}: \bz=(\bx,\bx,\ldots,\bx)\,\,\,\text{for some $\bx\in \mathbb{R}^{D}$} \},
\]
and $\bP_{\rmL_1^\perp}^\top\bU\neq \mathbf{0}$ means that $\bU_1, \bU_2, \cdots, \bU_K$ are not all the same.

Since $\tr(\bC\bH)=0$, we have $\|\bX\bU\|_F=0$. Let $\Sp(\bA)$ and $\Col(\bA)$ be the subspaces spanned by the row vectors of $\bA$ and the column vectors of $\bA$ respectively, then
\begin{align}
&\text{$\Sp(\bX)^\perp\supseteq \Col(\bU)$.}\label{eq:perp}
\end{align}

For any two subspace $\rmL$ and $\rmL'$, we use $\rmL+ \rmL'$ to represent the subspace $\{\bx+\by:\bx\in\rmL,\by\in\rmL'\}$, then we claim that to prove Theorem~\ref{thm:probablistic_generalize}, it is sufficient to prove the following statement with $p=\max\left((K-1)D-N,0 \right)$:  
\begin{align}
&\text{For generic $\bY\in\reals^{p\times KD}$, $\Col(\bU)\not\subseteq\Sp(\bY)+\rmL_1$ for every  $\bU\in\mathcal{U}$.}\label{eq:perp1}
\end{align}
The argument is as follows. Since $\sum_{i=1}^{K}\bX_i=\mathbf{0}$, \begin{align*}\dim(\Sp(\bX))=\rank(\bX)=&\rank([\bX_1,\bX_2,\cdots,\bX_{K-1},\bX_K])\\=&\rank([\bX_1,\bX_2,\cdots,\bX_{K-1}]),\end{align*} which is $\min\left((K-1)D, N \right)$ for generic $\{\bX_i\}_{i=1}^{K-1}\in\reals^{N\times D}$, so $\Sp(\bX)$ is a generic $\min\left((K-1)D, N \right)$-dimensional subspace in $\rmL_1^\perp$ (and $\rmL_1^\perp$ is a subspace of $\reals^{DK})$.

So, $\Sp(\bX)^\perp$ is the sum of $\rmL_1$ and a generic $p$-dimensional subspace in $\reals^{KD}$. Therefore, for generic $\{\bX\}_{i=1}^{K-1}$, $\Sp(\bX)^\perp$ is equivalent to $\Sp(\bY)+\rmL_1$, where $\bY$ is a generic matrix of size $p\times KD$. If \eqref{eq:perp1} holds, then \eqref{eq:perp} would not hold for generic $\{\bX\}_{i=1}^{K-1}$ and every $\bU\in\mathcal{U}$. By the analysis before \eqref{eq:perp}, the solution to the SDP problem is uniquely given by $(\bI,\ldots,\bI)^\top(\bI,\ldots,\bI)$, and Theorem~\ref{thm:probablistic_generalize} is proved.

\subsubsection{Step 2: finding matrices that do not satisfy \eqref{eq:perp1}}
In this part we show that every $\bY$ violating \eqref{eq:perp1} lies in the set $\cup_{d=1}^{p}\mathcal{M}_d$, where $\mathcal{M}_d$ is the range of the function\[
g_d\left(\{\bY_i\}_{i=0}^{K-1},\bZ,\rmL,\rmL_0\right)=\bP_{\rmL}\left[\left(\bY_1\bP_{\rmL_0}^\top,\cdots, \bY_K\bP_{\rmL_0}^\top\right)+
\left(\bY_0,\cdots, \bY_0\right)\right]+\bP_{\rmL^\perp}\bZ.
\]
The domain of the function $g_d$ is as follows: $\bY_i\in\reals^{d\times d}$ for $1\leq i\leq K-1$,  $\bY_0\in\reals^{d\times D}$, $\bZ\in\reals^{(p-d)\times KD}$, $\rmL_0$ is a $d$-dimensional subspace in $\reals^D$, and $\rmL$ is a $d$-dimensional subspace in $\reals^{p}$.  In addition, $\bY_K=-\sum_{i=1}^{K-1}\bY_i$.

For every $\bY$ violating \eqref{eq:perp1}, there exists $\bU\in\mathcal{U}$ such that $\Col(\bU)\subseteq\Sp(\bY)+\rmL_1$. We let $\bU=\bU^{(1)}+\bU^{(2)}$, where the columns of $\bU^{(1)}$ lie in $\rmL_1$ and the columns of $\bU^{(2)}$ are orthogonal to $\rmL_1$. As a result, $\Col(\bU^{(2)})$ is a subspace in $\reals^{KD}$ that intersects $\rmL_1$ only at origin, and we let $d=\dim(\Col(\bU^{(2)}))$. We have $d\geq 1$, since otherwise $\bU^{(2)}$ is a zero matrix, and $\bP_{\rmL_1^\perp}^\top\bU=\bP_{\rmL_1^\perp}^\top\bU^{(1)}=0$, which contradicts $\bU\in\mathcal{U}$.

Denote
\[
\bU^{(1)}=\left( \begin{array}{c}
\bU_1^{(1)} \\
\vdots \\
 \bU_K^{(1)} \end{array} \right)\,\,\,\text{and}\,\,\,\bU^{(2)}=\left( \begin{array}{c}
\bU_1^{(2)} \\
\vdots \\
 \bU_K^{(2)} \end{array} \right),
\]
then $\bU_1^{(1)} =\bU_2^{(1)} =\cdots=\bU_K^{(1)}$. Recall for all $1\leq i\leq K$, $\bU_i^{(1)}\bU_i^{(1)\,\top}+\bU_i^{(2)}\bU_i^{(2)\,\top}=\bU_i\bU_i^\top=\bI$, we have
\[
\bU_1^{(2)}\bU_1^{(2)\,\top} =\bU_2^{(2)}\bU_2^{(2)\,\top} =\cdots=\bU_K^{(2)}\bU_K^{(2)\,\top}.
\]
Therefore, \begin{equation}\Col\left(\bU_1^{(2)}\right)=\Col\left(\bU_2^{(2)}\right)=\ldots=\Col\left(\bU_K^{(2)}\right),\label{eq:U2_property}\end{equation} and \[\dim\left(\Col(\bU_1^{(2)})\right)=\rank\left(\bU_1^{(2)}\right)\leq \rank\left(\bU^{(2)}\right)=\dim\left( \Col(\bU^{(2)})\right)=d.\] 

Since $\Col(\bU)=\Col(\bU^{(2)})+\Col(\bU^{(1)})$, and $\Col(\bU^{(1)})\subseteq\rmL_1$, the assumption $\Sp(\bY)+\rmL_1\supseteq \Col(\bU)$ is equivalent to  $\Sp(\bY)+\rmL_1\supseteq \Col(\bU^{(2)})$. Recall $\dim\left(\Sp(\bY)\right)\leq p$, and $\Col(\bU^{(2)})$ is a subspace in $\reals^{KD}$ that intersects $\rmL_1$ only at origin, we have $d\leq p$.

Apply $\Sp(\bY)+\rmL_1\supseteq \Col(\bU^{(2)})$ and $\dim\left( \Col(\bU^{(2)})\right)=d$, there exists $\bY_0\in\reals^{d\times D}$ and $\rmL$, a $d$-dimensional subspace in $\reals^{p}$,  such that
\[
\Sp\left(\bP_{\rmL}^\top\left[\bY-\left(\bY_0,\bY_0,\cdots, \bY_0\right)\right]\right)=\Col\left(\bU^{(2)}\right).
\]
Recall the property \eqref{eq:U2_property}, there exists $\bY_i\in\reals^{d\times d}$ for $1\leq i\leq K$ and $\rmL_0$, a $d$-dimensional subspace in $\reals^D$ that contains $\Sp(\bU_{1}^{(2)\,\top})$, such that
\begin{equation}
\bP_{\rmL}^\top\left[\bY-\left(\bY_0,\bY_0,\cdots, \bY_0\right)\right]
=\left(\bY_1\bP_{\rmL_0}^\top,\bY_2\bP_{\rmL_0}^\top,\cdots, \bY_K\bP_{\rmL_0}^\top\right).
\label{eq:Y_property1}\end{equation}
In addition, since $\Col(\bU^{(2)})$ is orthogonal to $\rmL_1$. \begin{equation}\bY_K=-\sum_{i=1}^{K-1}\bY_i.\label{eq:Y_property2}\end{equation}

Combining \eqref{eq:Y_property1}, \eqref{eq:Y_property2}, and the estimation $1\leq d\leq p$, every $\bY$ that does not satisfy \eqref{eq:perp1} lies in the set $\cup_{d=1}^{p}\mathcal{M}_d$.

\subsubsection{Step 3: counting the dimension of $\mathcal{M}_d$}
In this step, we count the dimensions of $\mathcal{M}_d$  for all $1\leq d\leq p$, and show that they are smaller than $pKD$, the dimension of $\reals^{p\times KD}$, which implies that generic $\bY\in \reals^{p\times KD}$ does not belong to $\cup_{d=1}^p\mathcal{M}_d$, and \eqref{eq:perp1} is proved.

For any $d$, the degree of freedom is $\{\bY_i\}_{i=1}^{K-1}$ is $(K-1)d^2$, the degree of freedom of $\bY_0$ is $Dd$, the degree of freedom of $\bZ$ is $KD(p-d)$, the degree of freedom of $\rmL$ and $\rmL_1$ are $d(p-d)$ and $d(D-d)$ respectively. Considering that $d\leq p\leq D-1$, the total dimension of $\mathcal{M}_d$ is smaller than $KDp$, the dimension of $\reals^{p\times KD}$. Since $g_d$ is smooth and the dimension of its range is larger than its domain, all elements in $\mathcal{M}_d$ are its critical values. Applying~\cite[Theorem 6.8]{lee2003introduction}, $\mathcal{M}_d$ has measure zero in $\reals^{(D-1)\times KD}$. Therefore, $\reals^{(D-1)\times KD}\setminus\mathcal{M}_d$ is dense. Since $\reals^{(D-1)\times KD}\setminus\mathcal{M}_d$ is a closed set, generic $\bY$ does not lie in the set $\mathcal{M}_d$. Combining this result for all $1\leq d\leq p$, generic $\bY$ does not lie in the set $\cup_{d=1}^{p}\mathcal{M}_d$.
\end{proof}

\subsection{Proof of Lemma~\ref{lemma:pertubation}}\label{sec:lemmaproof}
We first state two lemmas that are rather easy to verify.
\begin{lemma}\label{lemma:frobenius}When $\bA\in\reals^{m\times n}$, $\bB\in\reals^{n\times l}$, where $l\geq n$ and the singular values of $\bB$ are $\sigma_1\geq \sigma_2\geq\cdots\geq \sigma_n$, then \begin{equation} \sigma_n\|\bA\|_F\leq \|\bA\bB\|_F\leq \sigma_1\|\bA\|_F.\label{eq:frobenius}\end{equation}\end{lemma}
\begin{proof}
First we claim that for any $\bx\in\reals^n$,
\begin{equation}\label{eq:frobenius1}
\sigma_n^2\|\bx\|^2\leq \|\bx^\top\bB\|^2\leq \sigma_1^2\|\bx\|^2.
\end{equation}
Assuming that the SVD decomposition of $\bB$ is given by $\bB=\bU_{\bB}\diag(\sigma_1,\cdots,\sigma_n)\bV_{\bB}^\top$, then the first inequality in \eqref{eq:frobenius1} can be proved as follows:
\[
\|\bx^\top\bB\|^2=\bx^\top\bB\bB^\top\bx=\bx^\top\bU_{\bB}\diag(\sigma_1^2,\cdots,\sigma_n^2)\bU_{\bB}^\top\bx\geq \sigma_n^2\|\bx^\top\bU_{\bB}\|^2=\sigma_n^2\|\bx\|^2,\]
and second inequality in \eqref{eq:frobenius1} can be proved similarly.

Assuming that $\bA=(\mathbf{a}_1,\ldots,\mathbf{a}_m)$, and combining  \eqref{eq:frobenius1} with $\bx=\mathbf{a}_i$ for $1\leq i\leq m$, \eqref{eq:frobenius} is proved.\end{proof}

\begin{lemma}\label{lemma:eigenvalue}
The smallest eigenvalue of
\[
\left( \begin{array}{cc}
a&b \\
b&0\end{array}\right)
\]
is smaller than
\[
-\frac{b^2}{\max(a,0)+|b|}.
\]
\end{lemma}
\begin{proof}
The smaller eigenvalue is \[\frac{a-\sqrt{a^2+4b^2}}{2}
=-\frac{2b^2}{a+\sqrt{a^2+4b^2}}\leq -\frac{2b^2}{a+(|a|+2|b|)}=-\frac{b^2}{\max(a,0)+|b|}.
\]
\end{proof}
\begin{proof}[Proof of Lemma~\ref{lemma:pertubation}]
First of all, WLOG we may assume that
\begin{equation}\label{eq:WLOG}
\bY=\left( \begin{array}{c}
\bI_{M\times M}  \\
\mathbf{0}_{M\times D-M} \end{array}\right).
\end{equation}
If we proved the case \eqref{eq:WLOG}, then other cases can be proved as follows. For generic $\bY$, there are invertible matrices $\bB\in\reals^{M\times M}$ and $\bA\in\reals^{D\times D}$ such that \[\bB\bY\bA=\left( \begin{array}{c}
\bI_{M\times M}  \\
\mathbf{0}_{M\times D-M} \end{array}\right).\] Note that
\[
(\bA^\top\bL\bB^{-1})(\bB\bY\bA)+(\bB\bY\bA)^\top(\bA^\top\bL\bB^{-1})^\top
=\bA^\top(\bL\bY+\bY^\top\bL^\top)\bA,
\]
we have
\[
\lambda_{\min}((\bA\bL\bB^{-1})(\bB\bY\bA)+(\bB\bY\bA)^\top(\bA\bL\bB^{-1})^\top)
\geq -\|\bA\|^2\epsilon,
\]
and similarly
\[
\lambda_{\min}((\bA\bL\bB^{-1})(\bB\bZ\bA)+(\bB\bZ\bA)^\top(\bA\bL\bB^{-1})^\top)
\geq -\|\bA\|^2\epsilon.
\]
Since the case \eqref{eq:WLOG} is assumed to be proved,
\[
\|\bA\bL\bB^{-1}\|_F\leq \epsilon\|\bA\|^2 c(\bB\bY\bA,\bB\bZ\bA).
\]
Applying Lemma~\ref{lemma:frobenius}, $\|\bA\bL\bB^{-1}\|_F\geq \|\bL\|_F\sigma_{\min}(\bA)/\|\bB\|$, and the generic case is proved with $c(\bY,\bZ)=\frac{\|\bA\|^2\|\bB\|}{\sigma_{\min}(\bA)} c(\bB\bY\bA,\bB\bZ\bA)$, which is positive for generic $\bZ$ since $c(\bB\bY\bA,\bB\bZ\bA)$ is positive for generic $\bB\bZ\bA$.

The rest of the proof will assume \eqref{eq:WLOG} and use induction on $k$, which is the integer such that $k(D-M)< M\leq (k+1)(D-M)$. For $k=1$, let us denote
\[
\bL=\left( \begin{array}{c}
\bL_{1}  \\
\bL_{2} \end{array}\right),\,\,\,\,\text{$\bZ=(\bZ_{1}, \bZ_{2})$,}
\]
where $\bL_1,\bZ_1\in\reals^{M\times M}$, $\bL_2\in\reals^{M\times (D-M)}$, and $\bZ_2\in\reals^{(D-M)\times M}$. Then \[
\bL\bY+\bY^\top\bL^\top=
\left( \begin{array}{cc}
\bL_{1}+\bL_{1}^\top  & \bL_{2}^\top\\
\bL_{2}  & \mathbf{0}\end{array}\right)
\]
If $\|\bL_2\|=b$ and $\|\bL_1\|=a$, then there exists $\bu\in\reals^{M},\bv\in\reals^{D-M}$ such that $\|\bu\|=\|\bv\|=1$ and $\bu^\top\bL_2\bv=b$. Then
\[
\left( \begin{array}{cc}
\bu  & 0\\
0 & \bv \end{array}\right)(\bL\bY+\bY^\top\bL^\top)\left( \begin{array}{cc}
\bu  & 0\\
0 & \bv \end{array}\right)^\top=\left( \begin{array}{cc}
\bu^\top(\bL_{1}+\bL_{1}^\top)\bu  & b\\
b  & 0 \end{array}\right),
\]
and its smallest eigenvalue is larger than $-\epsilon$ since $\left( \begin{array}{cc}
\bu  & 0\\
0 & \bv \end{array}\right)^\top\left( \begin{array}{cc}
\bu  & 0\\
0 & \bv \end{array}\right)=\bI$. Applying Lemma~\ref{lemma:eigenvalue} with the estimation  $\bu^\top(\bL_{1}+\bL_{1}^\top)\bu\leq 2a$, we have
\begin{equation}\label{eq:eigenvalue1}
\text{ $b^2\leq \epsilon(2a+b)$}.
\end{equation}
 In another aspect, applying Lemma~\ref{lemma:frobenius} we have
 \[
 \text{For generic $\bZ$, $C_1\|\bL_1\|_F\geq \|\bL_1\bZ_2\|_F\geq C_2\|\bL_1\|_F$ for some $C_1, C_2>0$.}
 \]
Therefore, we can find $\bu\in\reals^{M}$ and $\bv\in\reals^{N-M}$ such that  $|\bu^\top(\bL_1\bZ_2)\bv|\geq C_2 a$. Note that
\[
\left( \begin{array}{cc}
\bu  & 0\\
0 & \bv \end{array}\right)(\bL\bZ+\bZ^\top\bL^\top)\left( \begin{array}{cc}
\bu  & 0\\
0 & \bv \end{array}\right)^\top
=\left( \begin{array}{cc}
\bu^\top(\bL_1\bZ_1+\bZ_1^\top\bL_1^\top)\bu  & \bu^\top(\bZ_2^\top\bL_1^\top)\bv \\
\bu^\top(\bL_1\bZ_2)\bv & 0 \end{array}\right)+\bL',
\]
where $\|\bL'\|<C_3 b$. Since $|\bu^\top(\bL_1\bZ_1+\bZ_1^\top\bL_1^\top)\bu|\leq C_4 a$ for some $C_4>0$, Lemma~\ref{lemma:eigenvalue} shows that the smallest eigenvalue  of
\[
\left( \begin{array}{cc}
\bu^\top(\bL_1\bZ_1+\bZ_1^\top\bL_1^\top)\bu  & \bu^\top(\bZ_2^\top\bL_1^\top)\bv \\
\bu^\top(\bL_1\bZ_2)\bv & 0 \end{array}\right)
\]is smaller than
\[
-\frac{C_2^2}{C_4+C_1}a.\]
Let $C_4'=\frac{C_2^2}{C_4+C_1}$, applying $\lambda_{\min}(\bL\bZ+\bZ^\top\bL^\top)\geq -\epsilon$ we have
\begin{equation}\label{eq:eigenvalue1_2}
C_3b-C_4'a>-\epsilon.
\end{equation}
This means $a<\frac{1}{C_4'}(C_3b+\epsilon)$. Plug it into \eqref{eq:eigenvalue1}, we have
\[
b^2\leq \epsilon(b+2a)\leq\epsilon\Big(b+\frac{2}{C_4'}(C_3b+\epsilon)\Big)=\epsilon \Big(1+\frac{2C_3}{C_4'}\Big) b+\frac{2}{C_4'}\epsilon^2,
\]
which implies that $b<C\epsilon$. Applying \eqref{eq:eigenvalue1_2}, we also have $a<C'\epsilon$ and the Lemma is proved for $k=1$.

For $k>1$, let us first write
\[
\bL=\left( \begin{array}{cc}
\bL_{1,1}  & \bL_{1,2}\\
\bL_{2,1}  & \bL_{2,2}\end{array}\right),
\]
where $\bL_{1,1}\in\reals^{M\times (2M-D)}$, $\bL_{1,2}\in\reals^{M\times (D-M)}$, $\bL_{2,1}\in\reals^{(D-M)\times (2M-D)}$, $\bL_{2,2}\in\reals^{(D-M)\times (D-M)}$ and similarly write
\[
\bZ=\left( \begin{array}{cc}
\bZ_{1,1}  & \bZ_{1,2}\\
\bZ_{2,1}  & \bZ_{2,2}\end{array}\right),
\]
where $\bZ_{1,1}\in\reals^{(2M-D)\times M}$, $\bZ_{1,2}\in\reals^{(D-M)\times M}$, $\bZ_{2,1}\in\reals^{(2M-D)\times (D-M)}$, $\bZ_{2,2}\in\reals^{(D-M)\times (D-M)}$.
WLOG we may assume that $\bZ_{1,2}=\mathbf{0}$, by finding an appropriate orthogonal matrix $\bA\in\reals^{M\times M}$ and consider $(\bL\bA',\bA\bY,\bA\bY)$ instead of $(\bL,\bY,\bZ)$, then
\[
\bZ=\left( \begin{array}{cc}
\bZ_{1,1}  & \mathbf{0}\\
\bZ_{2,1}  & \bZ_{2,2}\end{array}\right),
\]
and for generic $\bZ$, $\bZ_{2,2}$ is full-rank, i.e., the rank is $D-M$.
Let $b=\|(\bL_{2,1},\bL_{2,2})\|_F$, $a_1=\|\bL_{1,1}\|_F$ and $a_2=\|\bL_{1,2}\|_F$. Then following the proof of \eqref{eq:eigenvalue1}, applying $\lambda_{\min}(\bL\bY+\bY^\top\bL^\top)\geq -\epsilon$ we have
\begin{equation}\label{eq:eigenvalue2}
\frac{b^2}{a_1+a_2+b}\leq \epsilon.
\end{equation}
If another aspect,
\[
\bL\bZ+\bZ^\top\bL^\top=\left( \begin{array}{cc}
\bL_{1,1}\bZ_{1,1}+\bL_{1,2}\bZ_{2,1}+\bZ_{1,1}^\top\bL_{1,1}^\top+\bZ_{1,2}^\top\bL_{2,1}^\top& \bL_{1,2}\bZ_{2,2}\\
\bZ_{2,2}^\top\bL_{1,2}^\top& \mathbf{0}\end{array}\right)+\bL',
\]
where $\|\bL'\|<C_1b$. Again Lemma~\ref{lemma:frobenius} means that we have $C_2a_2\leq \|\bL_{1,2}\bZ_{2,2}\|\leq C_3 a_2$ and $\|\bL_{1,1}\bZ_{1,1}+\bL_{1,2}\bZ_{2,1}+\bZ_{1,1}^\top\bL_{1,1}^\top+\bZ_{1,2}^\top\bL_{2,1}^\top\|\leq C_4 (a_1+a_2)$. Lemma~\ref{lemma:eigenvalue} then implies
\begin{equation}\label{eq:eigenvalue3}
\frac{C_2^2a_2^2}{C_4(a_1+a_2)+C_3a_2}\leq \epsilon+C_1b.
\end{equation}
At last, note that $\bL_{1,1}\bZ_{1,1}+\bL_{1,2}\bZ_{2,1}+\bZ_{1,1}^\top\bL_{1,1}^\top+\bZ_{1,2}^\top\bL_{2,1}^\top$ is a submatrix of $\bL\bZ+\bZ^\top\bL^\top$, so its smallest eigenvalue is also larger than $-\epsilon$. Since there exists $C_5$ such that $\|\bL_{1,2}\bZ_{2,1}+\bZ_{1,2}^\top\bL_{2,1}^\top\|\leq C_5 a_2$, the smallest eigenvalue of $\bL_{1,1}\bZ_{1,1}+\bZ_{1,1}^\top\bL_{1,1}^\top$ is larger than $-\epsilon-C_5a_2$. Let
\[
\bY_{1,1}=\left( \begin{array}{c}
\bI_{(2M-D)\times (2M-D)}  \\
\mathbf{0}_{(2M-D)\times (D-M)} \end{array}\right)
\]
and using the same argument, the smallest eigenvalue of $\bL_{1,1}\bZ_{1,1}+\bZ_{1,1}^\top\bL_{1,1}^\top$ is larger than $-\epsilon-C_5'a_2$.  Since $(k-1)(D-M)<2M-D\leq k(D-M)$, we may apply the case $k-1$ to $\bL_{1,1}$, $\bY_{1,1}$ and $\bZ_{1,1}$ and have
\begin{equation}\label{eq:eigenvalue4}
a_1\leq C_6((C_5+C_5')a_2+\epsilon).
\end{equation}

Plug in \eqref{eq:eigenvalue4} to \eqref{eq:eigenvalue2} and \eqref{eq:eigenvalue3} we have
\begin{equation}\label{eq:eigenvalue5}
{b^2}\leq \epsilon C_7(\epsilon+a_2+b)
\end{equation}
and
\begin{equation}\label{eq:eigenvalue6}
a_2^2\leq  C_8(\epsilon+a_2)(\epsilon+b),\,\,\text{i.e.,}a_2(a_2-C_8b)\leq \epsilon C_8(\epsilon+a_2+b).
\end{equation}

If $a_2\geq 2C_8b$, then \eqref{eq:eigenvalue6} implies
\[
\frac{1}{2}a_2^2\leq \epsilon C_8(\epsilon+a_2+\frac{1}{2C_8}a_2),
\]
which implies $a_2\leq C\epsilon$. This then implies $b\leq C'\epsilon$ (from assumption) and $a_1\leq C''\epsilon$ (from \eqref{eq:eigenvalue4}).

If $a_2< 2C_8b$, then \eqref{eq:eigenvalue5} implies $b<C\epsilon$, which then implies $a_1, a_2<C'\epsilon$.

For either case we have $a_1+a_2+b< C'''\epsilon$, and Lemma~\ref{lemma:pertubation} is proved since $\|\bL\|_F\leq \|\bL_{1,1}\|_F+\|\bL_{1,2}\|_F+\|(\bL_{2,1},\bL_{2,2})\|_F=a+b+c$.
\end{proof}

\section*{Acknowledgment}
The authors were partially supported by Award Number R01GM090200 from the NIGMS, FA9550-12-1-0317 and FA9550-13-1-0076 from AFOSR, LTR DTD 06-05-2012 from the Simons Foundation, and the Moore Foundation Data-Driven Discovery Investigator Award. The authors would like to thank the anonymous reviewers for their helpful comments and suggestions.\section*{References}
\bibliography{bib-online}

\end{document}